\documentclass[psamsfonts,12pt, reqno]{amsart}
\usepackage{graphicx}
\usepackage{amssymb,amsfonts}
\usepackage{amsmath, amscd}
\usepackage[all]{xy} 
\usepackage{enumerate}
\usepackage{mathrsfs}
\usepackage{amsthm}
\usepackage{eufrak}
\usepackage{fullpage} 
\usepackage{hyperref}
\usepackage{tikz-cd}

\newtheorem{thm}{Theorem}[section]
\newtheorem{cor}[thm]{Corollary}
\newtheorem{prop}[thm]{Proposition}
\newtheorem{lem}[thm]{Lemma}
\newtheorem{conj}[thm]{Conjecture}

\theoremstyle{definition}
\newtheorem{defn}[thm]{Definition}

\newtheorem{exmp}[thm]{Example}

\theoremstyle{remark}
\newtheorem{rem}[thm]{Remark}

\DeclareMathOperator{\gon}{gon}
\DeclareMathOperator{\Pic}{Pic}
\DeclareMathOperator{\irr}{irr}
\DeclareMathOperator{\Ext}{Ext}
\DeclareMathOperator{\Cliff}{Cliff}

\newcommand{\Sym}{{{\textrm{Sym}}}}

\newcommand{\length}{{{\textrm{length}\,}}}
\newcommand{\Ker}{{{\textrm{Ker}\,}}}

\renewcommand{\Im}{{{\textrm{Im}\,}}}

\newcommand{\G}{{{\mathbb{G}}}}

\renewcommand{\P}{{{\mathbb{P}}}}
\newcommand{\E}{{{\mathcal{E}}}}

\newcommand{\Z}{{{\mathbb{Z}}}}

\newcommand{\newword}[1]{\textbf{\emph{#1}}}
\newcommand{\str}{\mathcal{O}}

\newcommand{\PP}{\mathbb{P}}
\newcommand{\OO}{\mathcal{O}}
\newcommand{\ZZ}{\mathbb{Z}}

\newcommand{\arinj}{\ar@{^{(}->}}
\newcommand{\arsurj}{\ar@{->>}}

\numberwithin{equation}{section}

\title{The gonality of complete intersection curves}

\author{James Hotchkiss}
\address{Department of Mathematics\\ University of Michigan\\ 530 Church Street,
Ann Arbor, MI 48109}
\email{htchkss@umich.edu}

\author{Chung Ching Lau}
\address{Department of Mathematics, Statistics and Computer Science \\University of Illinois at Chicago, Chicago, IL 60607}
\email{clau22@uic.edu}

\author{Brooke Ullery}
\address{Department of Mathematics\\ Harvard University\\ 1 Oxford Street,
Cambridge, MA  02138}
\email{bullery@math.harvard.edu}

\begin{document}

\maketitle

\begin{abstract}
The purpose of this paper is to show that for a complete intersection curve $C$ in projective space (other than a few exceptions stated below), any morphism $f: C \to \P^r$ satisfying $\deg f^*\str_{\P^r}(1) <\deg C$ is obtained by projection from a linear space. In particular, we obtain bounds on the gonality of such curves and compute the gonality of general complete intersection curves. We also prove a special case of one of the well-known Cayley-Bacharach conjectures posed by Eisenbud, Green, and Harris.
\end{abstract} 

\section{Introduction}

Let $C$ be a complex projective curve. Recall that the \textit{gonality} of $C$, $\gon(C)$, is the minimum degree of a surjective morphism 
$$\tilde{C} \longrightarrow \P^1,$$ where $\tilde{C}$ is the normalization of $C$. Thus, $C$ is rational precisely when $\gon (C) = 1$, and, more generally, gonality measures how far the curve is from being rational. Gonality is a classical invariant, and there has been significant interest in bounding the gonality of various classes of curves  and characterizing the corresponding maps to $\P^1$. Specifically, if $C$ is embedded in projective space, it is natural to ask whether the gonality is related to the embedding of the curve.

For example, the gonality of plane curves (i.e. complete intersection curves of codimension one) is well understood. If $C \subset \P^2$ is a smooth curve of degree $d$, then the map $C \to \P^1$ obtained by projecting from a point in $C$ has degree $d-1$, giving an upper bound on the gonality. In fact, a classical theorem of Noether states that if $d \geq 3$, then $$\gon(C) = d-1,$$ and any covering of $\P^1$ of degree $d-1$ is obtained by projecting from a point.

More generally, if $C$ is a smooth degree $d$ curve in projective space and $V$ a base point free linear system of degree at most $d$, one may ask when the morphism defined on $C$ by $V$ is obtained by projection from a linear subspace of codimension $\dim(V) +1$. This question was first studied for complete intersection curves in $\P^3$ by Ciliberto and Lazarsfeld \cite{CL} and later by Basili \cite{Bas}. The former authors studied this question for $\dim(V) =2$, the latter for $\dim(V) =1$. Specifically, Basili showed that if $C \subset \P^3$ is a smooth complete intersection curve, then the gonality is indeed computed by projection from a line and every minimal covering arises in this way. Recently, Hartshorne and Schlesinger generalized Basili's results to smooth ACM curves in $\P^3$ satisfying some assumption of generality, with the exception of a few cases \cite{HS}. The same result holds for many other specific classes of curves in $\P^3$ (e.g. \cite{Bal}, \cite{EF}, \cite{Far}, \cite{Har}, \cite{Mar}). See \cite{HS} for a detailed review on curves in $\P^3$ whose gonality is computed by projection from a line.

At the other extreme, i.e. higher dimensional hypersurfaces, the recent paper \cite{BDELU} calculated the so-called degree of irrationality of very general hypersurfaces in projective space. If $X$ is a smooth variety of dimension $n$, then the \textit{degree of irrationality} of $X$, $\irr(X)$, is the minimum degree of a dominant rational map $X \dashrightarrow \P^n$. Clearly, this definition agrees with the definition of gonality in the $n=1$ case, and $\irr(X) = 1$ if and only if $X$ is rational. The main result of \cite{BDELU} states that if $X \subset \P^{n+1}$ is a very general hypersurface of degree $d \geq 2n+1$, then $\irr(X) = d-1$, and if $d \geq 2n+2$, then any dominant map $X \dashrightarrow \P^n$ is obtained by projecting from a point on $X$.

Returning to curves, the main result about gonality of complete intersection curves in higher dimensional projective spaces to date is a lower bound due to Lazarsfeld: 
\begin{thm}[cf. {{{\cite[Exercise 4.12]{Laz}}}}]\label{Laz}
 Let $C \subset \P^{n}$ be a smooth complete intersection curve of type 
$(a_1, a_2, \ldots, a_{n-1}),$ where $2 \leq a_1 \leq \cdots \leq a_{n-1}$. Then $$\gon(C) \geq (a_1 - 1)a_2 \cdots a_{n-1}.$$ \end{thm}
\noindent  In light of the previous examples, one may ask whether every such map is given by projection. Our main result confirms this in a more general setting as long as $C$ satisfies mild degree restrictions:

\begin{thm}\label{thma}
	Let $C \subset \P^{n}$ be a complete intersection curve of type $(a_1, \dots, a_{n-1 })$, with $$4 \leq a_1 < a_2 \leq \cdots \leq a_{n-1}.$$ Then for $r < n$, any morphism $f: C \to \P^r$ satisfying $$\deg f^* \str_{\P^r}(1) < \deg (C)$$ is obtained by projecting from an $(n-r-1)$-plane. Thus $\gon(C) = \deg (C) - \gamma$, where $\gamma$ is the maximum number of points on $C$ contained in an $(n-2)$-plane.
\end{thm}

\begin{rem}
In fact, we can weaken the hypotheses of Theorem \ref{thma} slightly. As long as $C$ lies on a smooth complete intersection threefold of type $(w_3, \cdots,  w_{n-1})$ and the remaining two degrees $a$ and $b$ cutting out $C$ satisfy $4 \leq a <b$, then the conclusion of the theorem holds.
\end{rem}
For very general complete intersection curves, the same conclusion of Theorem \ref{thma} holds even if we relax the bounds on the degrees; we only require that the largest degree is at least 4, and the sum of the remaining degrees is at least $n+1$. See Theorem \ref{thma1} for the full statement.


By studying the secancy of $(n-2)$-planes to complete intersection curves in Section \ref{n-2}, we are able to give a bound on the gonality of arbitrary complete intersection curves.
\begin{thm}\label{thmb}
Let $C\subset \mathbb{P}^n$ be a complete intersection curve of type $(a_1,\cdots,a_{n-1})$, with $2\leq a_1\leq a_2\leq \cdots \leq a_{n-1}$.
Then $C$ has an $(n-2)$-plane which is at least  $(2n-2)$-secant  to $C$ unless
$$n=3 \ \mbox{and} \ (a_1,a_2)=(2,2), (2,3), (3,3); \ \mbox{or} \ n=4 \  \mbox{and} \ (a_1,a_2,a_3)=(2,2,2).$$
In particular, $\gon(C)\leq \deg(C)-2n+2$.
\end{thm}

If we make a stronger assumption on the degrees, we can compute the value of $\gamma$ from Theorem \ref{thma} for $C$ general. This leads to the following theorem. (One may find a more comprehensive statement in Theorem \ref{thmmain}.)

\begin{thm}\label{corb}
	Let $C \subset \P^n$ be a general complete intersection curve of type $(a_1, \ldots, a_{n-1})$, with $$ \max\{4, n-1\} \leq a_1 <  a_2 \leq \cdots \leq a_{n-1}$$ and $n \geq 4$.
	 Then $$ \gon(C) = \deg(C) - 2n +2.$$
	 Moreover, there are only finitely many $(2n-2)$-secant $(n-2)$-planes, and each of them intersects $C$ at $2n-2$ distinct points.
	 Furthermore, the Clifford index of $C$ equals $\gon(C)-2=\deg(C) - 2n$.
\end{thm}

The proof of Theorem \ref{thma} relies on a generalization of the classical Noether-Lefschetz Theorem \cite{Lef}, which we prove in Section 3. It states that, under certain degree restrictions, a complete intersection curve in projective space lies on a complete intersection surface with Picard group generated by the hyperplane class. However, Theorem \ref{thma} also holds in the more general setting of arbitrary curves lying on surfaces with Picard group $\Z$:

\begin{thm}\label{thmc}
	    Let $C \subset \P^{n}$ be a smooth non-degenerate curve lying on a smooth surface $S$ with $\Pic(S) = \Z \cdot [ \str_S (1)]$, and $C \in |\str_S(\alpha)|$, where $\alpha \geq 4$. Then any morphism $f: C \to \P^r$ with $r<n$ satisfying $$\deg f^* \str_{\P^r}(1) < \deg (C)$$ is obtained by projecting from an $(n-r-1)$-plane.
	\end{thm}
	
\begin{rem} In \cite{Ras}, Rasmussen independently showed that, under stronger hypotheses, the fibers of a morphism $C \to \P^1$ of degree less than the degree of $C$ must lie in hyperplanes, which follows from our Lemma \ref{lem:key lemma}.
\end{rem}

\begin{rem}
A special case of one of the Cayley-Bacharach conjectures posed by Eisenbud, Green, and Harris (\cite[Conjecture CB12]{EGH}) follows easily from the proof of Theorem \ref{thmc}. We discuss this in Section \ref{sec:CB}.
\end{rem}
	
We also obtain a bound on the gonality in this more general setting:
	
\begin{cor}\label{cord} 
	With $C$ and $S$ as in Theorem \ref{thmc}, $$ \deg(C) - \deg(S) \leq \gon(C).$$ If, in addition, $C$ is linearly normal, then $$\gon(C) \leq \deg(C) - 2n+3.\footnote{
Asher Auel and Dave Jensen showed us how to improve the upper bound from an earlier version of this paper by applying a result of Coppens and Martens \cite{CM}.}$$
\end{cor}

The authors would like to thank Izzet Coskun and Lawrence Ein for being very generous in sharing their ideas and expertise in this area.
Additionally, we would like to thank Asher Auel, Ciro Ciliberto, Joe Harris, Dave Jensen, Hannah Larson, Rob Lazarsfeld, Jake Levinson, Ian Shipman, David Stapleton, and Isabel Vogt for many helpful comments and conversations. 
The research of the second author was supported by a Croucher Foundation Postdoctoral Fellowship.
The research of the third author was partially supported by an NSF Postdoctoral Fellowship, DMS-1502687.

The paper is organized as follows.
Section 2 is a short exposition on the Cayley-Bacharach condition, which we use in proving Theorem \ref{thmc}.
In Section 3, we prove a generalization of the classical Noether-Lefschetz Theorem.
In Section 4, we prove Theorems \ref{thma} and \ref{thmc} and Corollary \ref{cord}.
In Section 5, we prove several results about secancy of $(n-2)$-planes to complete intersection curves, which lead to proofs of Theorems \ref{thmb} and \ref{thmc}.
In Section 6, we modify the proof of a lemma in Section 4 to prove a case of a Cayley-Bacharach conjecture posed by Eisenbud, Green, and Harris. 

Concerning conventions, we work throughout over the complex numbers, and we will often switch between divisor and line bundle notation.

\section{The Cayley-Bacharach condition}

Suppose $Z$ is a set of distinct points on a smooth variety $X$ of dimension $n$. Let $L$ be a line bundle on $X$. The set of points $Z$ \textit{satisfies the Cayley-Bacharach condition with respect to}  the complete linear system $|L|$ if every section of $L$ vanishing at all but one of the points of $Z$ also vanishes at the remaining point.

\begin{rem} If $Z$ satisfies the Cayley-Bacharach condition with respect to $|L|$, then $Z$ fails to impose independent conditions on $L$. However, the converse is not true. For instance, if $P_1, P_2, P_3 \in \P^2$ lie on a line $\ell$, and $P_4 \notin \ell$, then the set of points $\{P_1 , P_2, P_3, P_4\}$ does not impose independent conditions on $|\str_{\P^2}(1)|$, but it doesn't satisfy the Cayley-Bacharach condition with respect to  $|\str_{\P^2}(1)|$.
\end{rem}

In the case where $Z$ is a general fiber of a generically finite rational map $X \dashrightarrow \P^n$, by analyzing the trace map, one obtains the following, a special case of \cite[Proposition 4.2]{Bas2}:

\begin{thm}{\cite{Bas2}}
Let $X$ be a smooth variety of dimension $n$, and $X \dashrightarrow \P^n$ a generically finite dominant rational map. Let $Z \subset X$ be a finite reduced fiber. Then $Z$ satisfies the Cayley-Bacharach condition with respect to the canonical linear system $|K_X|$.
\end{thm}

If the canonical bundle of $X$ is sufficiently positive, this forces various geometric constraints on the fibers. For instance, if $X$ is a hypersurface, a simple geometric argument shows that under certain degree hypotheses the above theorem implies that the general fiber must be collinear (cf. \cite{BCD}). The authors of \cite{BDELU} exploited this fact to compute the degree of irrationality of very general hypersurfaces.

For our purposes, we will only be dealing with the case in which $X$ is a curve. However, in this case, every such map is actually a morphism, and a much more general result follows easily from the Riemann-Roch Theorem:

\begin{thm}
\label{CBcurves}
Let $C$ be a smooth curve and $f: C \to \P^r$ a morphism. Then any reduced divisor $Z \in |f^*\str_{\P^r}(1)|$ satisfies the Cayley-Bacharach condition with respect to $|K_C|$.
\end{thm}

\begin{proof}
The complete linear system $|Z|$ is base-point free, so for any $P \in Z$, $$\dim|Z - P| = \dim|Z| -1.$$ Applying Riemann-Roch, we get $$\dim|K_C -Z | = \dim|K_C-(Z-P)|,$$ which is equivalent to the Cayley-Bacharach property.
\end{proof}

\section{A generalization of the Noether-Lefschetz theorem}

In order to prove Theorem \ref{thma}, we need to show that a complete intersection curve  lies on a complete intersection surface whose Picard group is generated by the hyperplane class. More precisely, the goal of this section is to prove the following theorem.

	\begin{thm}
	\label{thm:main theorem}
	    Let $W$ be a smooth, complete intersection threefold in $\mathbb{P}^n$ with $n \geq 4$ of type $(w_4, \dots, w_n)$, and let $C$ be a smooth complete intersection curve in $W$ of type $(d - e, d, w_4, \dots, w_n)$. If $4 \leq d - e < d$, then the very general complete intersection surface $S$ containing $C$ of type $(d, w_4, \dots, w_n)$ is smooth and satisfies $\operatorname{Pic} S = \mathbb{Z} \cdot [\mathcal{O}_S(1)]$.
	\end{thm}

	If the surface of type $(d - e, w_4, \dots, w_n)$ containing $C$ is smooth, then we recover a special case of \cite[Theorem III.2.1]{Lop}. Here we will deal with the case when it is possibly singular. 
	
\begin{rem}
\label{assumption remark}
The assumption that $n \geq 4$ is a convenience---the statement can be extended to $\mathbb{P}^3$ by augmenting Proposition \ref{lem:crucial lemma} below with \cite[Lemmas II.3.3' and II.3.3'']{Lop}. 
\end{rem}

Following Lopez's approach for proving \cite[Theorem II.3.1]{Lop}, we will restrict our attention to the main technical ingredient (Corollary \ref{cor:picard group of central fiber}) of Theorem \ref{thm:main theorem} and refer to \cite{GH3} for the remainder of the argument, which carries through without incident. 

	Let $T$ be the unique surface of type $(d - e, w_4, \dots, w_n)$ containing $C$, and let $P$ be the very general surface of type $(e, w_4, \dots, w_n)$ not containing $C$. Let $X$ be the total space of the pencil interpolating $P \cup T$ and $S$; we will regard $X_0 = P \cup T$ as the central fiber, which is singular at the singularities of $T$ and along the double curve $D = P \cap T$.

	We may choose $P$ to meet $S$ and $T$ transversely, which leaves $X$ with ordinary double point singularities at the finite intersection 
	\[
		P \cap T \cap S = D \cap C = \{p_1, \dots, p_N\}.
	\]
	Blowing up each of the $p_i$ produces a smooth model of $X$ whose central fiber has a quadric surface $Q_i$ over each singular point. The strict transform of $T$ specifies a ruling of each $Q_i$, and blowing down along each of the specified rulings yields a family $\tilde X$ with smooth total space. Away from the central fiber, $X$ and $\tilde X$ are isomorphic, but the central fiber of $\tilde X$ is a reducible surface $\tilde X_0 = \tilde P \cup \tilde T$, where $\tilde P$ is the blowup of $P$ at each of the $p_i$, and $\tilde T \cong T$. $\tilde T$ and $\tilde P$ meet in a double curve $\tilde D \cong D$.

	Since $\tilde T$ and $\tilde P$ meet transversely, 
	\begin{equation} \label{fibered product}
		\operatorname{Pic} \tilde X_0 = \operatorname{Pic} \tilde T \times_{\operatorname{Pic} \tilde D} \operatorname{Pic} \tilde P. \tag{$\star$}
	\end{equation}
	and our goal is to compute $\operatorname{Pic} \tilde X_0$. Consider the restrictions of $\operatorname{Pic} \tilde T$ and $\operatorname{Pic} \tilde P$:
	\[
		\begin{tikzcd}[column sep=tiny]
			\operatorname{Pic} \tilde T \ar[dr, "r_1"'] & & \operatorname{Pic} \tilde P \ar[dl, "r_2"] \\
			& 	\operatorname{Pic} \tilde D
		\end{tikzcd}
	\]
	The main ingredient is the following proposition:
	\begin{prop}
	\label{lem:crucial lemma}
	    Let $E_i$ be the exceptional divisor in $\tilde P$ over $p_i$, and let $\mathcal{O}_{\tilde P}(1)$ be the pullback of the hyperplane class on $P$ to $\tilde P$. Then
	    \begin{enumerate} [(i)]
	    	\item $\operatorname{Pic} P = \mathbb{Z} \cdot [\mathcal{O}_P(1)]$
	    	\item $\operatorname{Ker} r_2 = \mathbb{Z} \cdot [\mathcal{O}_{\tilde P}(\sum E_i) \otimes \mathcal{O}_{\tilde P}(-d)]$
	    	\item $\operatorname{Im} r_1 \cap \operatorname{Im} r_2 = \mathbb{Z} \cdot [\mathcal{O}_{\tilde D}(1)]$
	    	\item $r_1$ is injective.
	    \end{enumerate}
	\end{prop}

	\begin{proof}
		By our assumptions on the degrees, (i) follows from the classical Noether-Lefschetz theorem, and (ii) follows from (i). Moreover, (iii) follows from a standard monodromy argument which is given in  \cite[Lemma II.3.3 and Subclaim II.3.4]{Lop}, and it remains to show (iv). 

	    First, note that the singularities of $T$ are isolated, since $C$ is smooth and moves in $T$. Furthermore, $T$ is a complete intersection and hence has, for instance, Cohen-Macaulay singularities, so $T$ is normal. Notationally, since $\tilde T \cong T$ and $\tilde D \cong D$, we will work with $T$ and $D$ for simplicity.

		 Lopez shows in \cite[Lemma II.2.4]{Lop} that unless $e = 1$ and $T$ is either ruled by lines, the Veronese surface, or its general projection to $\mathbb{P}^4$ or $\mathbb{P}^3$, then there exists a pencil $|\mathscr{V}|$ of irreducible curves within $|\mathcal{O}_T(e)|$. If $T$ is the Veronese surface or its general projection to $\mathbb{P}^4$, $r_1$ is clearly injective. The general projection of the Veronese surface to $\mathbb{P}^3$ (the Steiner surface) is not normal---see \cite[pg. 632]{GH2}---and we have excluded $n = 3$.

		Therefore, if we assume that $T$ is not ruled, then we may assume that $T$ possesses a pencil $|\mathscr{V}|$ of irreducible curves within $|\mathcal{O}_T(e)|$. Let $f:T' \to \mathbb{P}^1$ be a resolution of $|\mathscr{V}|$:
	    \[
	    	\begin{tikzcd}[column sep=small]
	    		& T' \ar[dl, "\epsilon"'] \ar[dr, "f"] \\
	    		T \ar[rr, dotted] & & \mathbb{P}^1
	    	\end{tikzcd}
	    \]
	    Let $L$ be an element of $\operatorname{Pic} T$, and assume that $L$ restricts trivially to a general member of $|\mathscr{V}|$. Then $\epsilon^* L$ restricts trivially to every fiber of $f$ over an open subscheme $U$ of $\mathbb{P}^1$, and cohomology and base change implies that $\epsilon^* L|_{f^{-1}(U)}$ is actually the pullback of an invertible sheaf on $\mathbb{P}^1|_U$. Since there are finitely many fibers in the complement of $f^{-1}(U)$, all of which are irreducible, $\epsilon^* L$ is globally the pullback of a line bundle $\mathcal{O}_{\mathbb{P}^1}(\ell)$ on $\mathbb{P}^1$. 

	    But $f^* \mathcal{O}_{\mathbb{P}^1}(\ell) \cong \ell(eH - E)$, where $E$ is supported on the exceptional divisor $\operatorname{Exc}(\epsilon)$ of $\epsilon$. On the complement of $\operatorname{Exc}(\epsilon)$, $\epsilon^* L \cong L \cong  \ell eH$, and since $T$ is normal, the isomorphism $L \cong \ell e H$ extends across all of $T$. It follows that $L$ is trivial, and the restriction map $\operatorname{Pic} T \to \operatorname{Pic} C$ is injective for the very general curve $C$ in $|\mathscr{V}|$.  

	    Next, assume that $T$ is ruled by lines and $e = 1$. If $T$ is smooth, then it cannot be ruled---it would be of general type. So $T$ possesses some singularities, and the upshot is that $T$ must be a cone:

	    \begin{lem}
	    \label{lem:cone lemma}
	        Let $T$ be a singular, normal, irreducible surface in $\mathbb{P}^n$ which is ruled by lines. Then $T$ is a cone over a smooth, degenerate curve. 
	    \end{lem}

	    \begin{proof}
	        Let $Z$ be a connected component of the curve in $\mathbb{G}(1, n)$ which sweeps out $T$, and let $Z'$ be its normalization. Note that $Z$ is irreducible, as otherwise $T$ would be singular along a line, and likewise the normalization $Z' \to Z$ is bijective. The universal line $\Phi \to Z$ pulls back to a family of lines $\Phi' \to Z'$, and there is a natural map $\pi: \Phi' \to T$. 

	        First, we claim that $\pi$ is birational. Assume, for the sake of contradiction, that $\deg \pi \geq 2$, and let $\Lambda$ be the line corresponding to an arbitrary point on $Z$. For every point $p$ along $\Lambda$, there is an additional line on $T$ meeting $\Lambda$ at $\pi$, and since $Z$ is irreducible, it follows that every line which comes from $Z$ meets $\Lambda$. Applying the same argument to three distinct $\Lambda$ yields that $T$ is a plane, which contradicts various of the hypotheses on $T$.

	        Second, we claim that $\pi$ contracts a curve. This follows from the birationality of $\pi$, the normality of $T$, and Zariski's Main Theorem, as well as the fact that $\pi$ cannot be an isomorphism. Let $E$ denote a curve contracted by $\pi$. By definition of $\Phi'$, $E$ cannot be supported on the fibers of $\Phi' \to Z'$, so $E$ must meet every fiber. Then $\pi(E)$ lies on every line which comes from $Z$, and by taking a general hyperplane $H \cap T$ section of $T$, we see that $T$ may be regarded as the cone over $H \cap T$ with vertex $\pi(E)$.
	    \end{proof}

		To conclude the proof of Proposition \ref{lem:crucial lemma}, we make the following observations:
		\begin{enumerate} [(i)]
			\item $\Pic(T) = \operatorname{Pic}(T - \pi(E))$, since $T$ is normal, so to show that $r_1$ is injective, it suffices to show that $\operatorname{Pic}(T - \pi(E)) \to \operatorname{Pic}(D)$ is injective.  
			\item Let $f:D \to T \cap H$ be the projection away from $\pi(E)$. Since the $f$ is finite of degree $d_P = \deg P$, there is a norm map $\mathcal{N}_{f}:\operatorname{Pic}(D) \to \operatorname{Pic}(T \cap H)$, which satisfies the property that the composition
			\[
			\begin{tikzcd}
				\operatorname{Pic}(T \cap H) \ar[r, "f^{*}"] &  \operatorname{Pic}(D)  \ar[r, "\mathcal{N}_{f}"] & \operatorname{Pic}(T \cap H)
			\end{tikzcd}
			\]
			is given by $L \mapsto L^{\otimes d_P}$ for any $L \in \operatorname{Pic}(T \cap H)$ (cf. \cite[\href{https://stacks.math.columbia.edu/tag/0BCX}{Tag 0BCX}]{stacks-project}). This implies that any element of $\Ker(f^*)$ is $d_P$-torsion in $\operatorname{Pic}(T \cap H)$.
			\item $\operatorname{Pic}(T)$ is torsion-free, e.g. by \cite[Theorem, pg. 170]{Bad}.
		\end{enumerate}
		Putting things together, let $f':T \to T \cap H$ denote the projection, which has a natural section given by the inclusion $T \cap H \to T$ and gives rise to a commutative diagram
		\[
			\begin{tikzcd}
				\operatorname{Pic}(T) = \operatorname{Pic}(T - \pi(E)) \ar[d, hook] \ar[r] & \operatorname{Pic}(D) \\
				\operatorname{Cl}(T) \ar[d, "\sim"] \\
				\operatorname{Pic}(T \cap H) \ar[uur, "f^*"] \ar[uu, bend left = 40, "(f')^*"]
			\end{tikzcd}
		\]
		We refer to  \cite[Exercise II.6.3]{Har2} for the isomorphism $\operatorname{Cl}(T) \cong \operatorname{Pic}(T \cap H)$. The map $\operatorname{Pic}(T) \to \operatorname{Cl}(T)$ is injective since $T$ is normal. 

		To conclude, given $L \in \operatorname{Pic}(T)$, the pullback of $L$ to $\Pic(T \cap H)$ is torsion-free (since $L$ is torsion-free, by (iii), and the pullback is injective), so following the diagram we see that the image of $L$ in $\operatorname{Pic}(D)$ is nonzero, by (ii). 
	    \end{proof}

		\begin{cor}	
		\label{cor:picard group of central fiber}
		    $\operatorname{Pic} \tilde X_0 = \mathbb{Z} \cdot [\mathcal{O}_{\tilde X_0}(1)] \oplus \mathbb{Z} \cdot [M]$, where 
		\begin{align*}
			M|_{\tilde T} &=  \mathcal{O}_{\tilde T} \\
			 M|_{\tilde P} &= \mathcal{O}_{\tilde P}(e) \otimes \mathcal{O}_{\tilde P}(\tilde D))
		\end{align*}
		\end{cor}

		\begin{proof}
		    By the discussion preceding Proposition \ref{lem:crucial lemma}, we know that $\operatorname{Pic} \tilde X_0 = \operatorname{Pic}(\tilde T) \times_{\operatorname{Pic}(\tilde D)} \operatorname{Pic}(\tilde P)$. Applying the various statements of Proposition \ref{lem:crucial lemma}, we simply need to show that $M|_{\tilde P} = \mathcal{O}_{\tilde P}(\sum E_i) \otimes \mathcal{O}_{\tilde P}(-d)$. This follows from the observation that $\mathcal{O}_{\tilde P}(\tilde D + \sum E_i) = \mathcal{O}_{\tilde P}(d - e)$. 
		\end{proof}

		Note that by our description of the Picard group in \eqref{fibered product}, it suffices to specify $M$ by specifying its restrictions to the components of $\tilde X_0$. The remainder of the proof of Theorem \ref{thm:main theorem} follows the argument given in \cite[pg. 37-39]{GH3}, which carries through in this setting without revision. 



\section{Proofs of the main theorems}

In this section, we first prove Theorem \ref{thmc} and then use it to prove Theorem \ref{thma}. The structure of the proof of Theorem \ref{thmc} is as follows: we take a surface of Picard rank one containing the curve, and apply the following theorem of Griffiths and Harris to construct a vector bundle on the surface.
\begin{thm}{{{\cite[Proposition 1.33]{GH}}}}
	\label{thm:GH}
Let $S$ be a smooth projective surface, $L$ a line bundle on $S$, and $Z \subset S$ a reduced set of points. Then there exists a rank two vector bundle $\E$ with $\det \E = L$ along with a section $s \in H^0(\E)$ with $Z(s)=Z$ if and only if $Z$ satisfies the Cayley-Bacharach property with respect to $|K_S+L|$.
\end{thm}

\noindent We then determine that the vector bundle is Bogomolov unstable, which gives a lower bound on the degree of the fiber and forces the fiber to be contained in a hyperplane. Analyzing the geometry, we conclude that each hyperplane must contain a single fiber, and that the hyperplanes lie in a linear pencil.

Combining the Griffiths-Harris theorem with Theorem \ref{CBcurves}, we easily obtain the following.

\begin{lem}
	\label{lem:CB}
Let $C$ be a smooth curve satisfying the assumptions of Theorem \ref{thmc}. Let $f: C \to \P^r$ be a morphism, and let $\Gamma \in |f^*\str_{\P^r}(1)|$ be general (and thus reduced). Then there is a rank two vector bundle $\E$ on $S$ sitting in the short exact sequence 
	\begin{equation}
	\label{SES1} 
	0 \to \str_S \to \E \to \mathcal{I}_{\Gamma, S}(\alpha) \to 0.
	\end{equation}
\end{lem}

\begin{proof}

By Theorem \ref{CBcurves}, $\Gamma$ satisfies the Cayley-Bacharach condition with respect to  the canonical linear series $|K_C|.$ In particular,  by adjunction, it satisfies the Cayley-Bacharach condition with respect to $|K_S + \alpha H|$ where $H$ is a hyperplane section on $S$. 

We then apply Theorem \ref{thm:GH} with $L = \str_S(\alpha)$ to obtain $\E$, along with a global section $s$ that vanishes precisely along $\Gamma$. A modification of the Koszul complex associated to $\E$ and $s$ (see, for example, \cite[page 320]{Laz2}) yields the exact sequence
$$0 \longrightarrow \str_S \stackrel{\cdot s}{\longrightarrow} \E \longrightarrow \wedge^2 \E \longrightarrow \wedge^2 \E \otimes \str_{\Gamma}\longrightarrow 0.$$
Since $s$ vanishes along $\Gamma$ and $\wedge^2 \E \cong \str_S(\alpha)$, the cokernel of the map $\E \to \wedge^2 \E$ is $\mathcal{I}_{\Gamma, S}(\alpha)$, which gives the desired short exact sequence.

\end{proof}

Using \eqref{SES1}, we can check that $\E$ is Bogomolov unstable, producing a second representation of $\E$ as extension. The plan is to compare the two.  We first recall Bogomolov's Instability Theorem.

\begin{thm}[cf. {{{\cite[Corollary 2 in Section 10.12]{Bog}}}}]
	\label{thm:Bog}
Let $\mathcal{F}$ be a rank two vector bundle on a smooth projective surface $X$. If 
	\begin{equation}
	\label{eqn:Bog}
	c_1(\mathcal{F})^2 - 4 c_2(\mathcal{F}) > 0,
	\end{equation}
then $\mathcal{F}$ is \newword{Bogomolov unstable}. That is, there exists a finite subscheme $Z \subset X$ (possibly empty), plus line bundles $L$ and $M$ on $X$ sitting in an exact sequence 
	\begin{equation} 
	\label{eqn:BogSES}
	0 \to L \to \mathcal{F} \to M \otimes \mathcal{I}_Z \to 0
	\end{equation}
where $(L-M)^2>0$ and $(L-M)A >0$ for all ample divisors $A$.
\end{thm}

We can now use this along with \ref{SES1} to prove our key lemma. The technique of the proof is similar to that of Reider's Theorem (cf. \cite[Theorem 2.1]{Laz}  or \cite{Rei} for Reider's original proof).

\begin{lem}
	\label{lem:key lemma}
Let $C$ be a smooth curve satisfying the assumptions of Theorem \ref{thmc}. Let $f:C \to \P^r$ be a morphism with $r<n$. Suppose $\deg f^*\str_{\P^r}(1) < d_C := \deg(C)$, and let $\Gamma \in |f^*\str_{\P^r}(1)|$ be a divisor. Then 
	\begin{enumerate}
\item $\Gamma$ lies in a hyperplane, and 
\item $\deg f^*\str_{\P^r}(1) \geq \deg(S) \cdot (\alpha-1)$.
	\end{enumerate}
\end{lem}

\begin{proof}
Without loss of generality, we can take $\Gamma$ to be general in $|f^*\str_{\P^r}(1)|$. Let $\E$ be the vector bundle on $S$ obtained in Lemma \ref{lem:CB}. First, we show that $\E$ is Bogomolov unstable. By \eqref{SES1}, the Chern classes of $\E$ are given by 
		\begin{align*}
		 	c_1(\E) &= \alpha H, \\
			c_2(\E) &= d_{\Gamma},
		\end{align*} 
where $d_{\Gamma}$ is the length of $\Gamma$, i.e. $\deg f^*\str_{\P^r}(1)$.
		Let $d_S$ be the degree of $S$. Then 
		\[
			 c_1(\mathcal{E})^2 - 4 c_2(\mathcal{E}) = \alpha^2 d_S - 4 d_{\Gamma},
		\]
which greater than zero since $d_{\Gamma} < d_C = d_S \alpha$, and $\alpha \geq 4$. Thus, 
$\E$ sits in the short exact sequence 
	\begin{equation}
	\label{SES2}
	0 \to L \to \E \to M \otimes \mathcal{I}_Z \to 0
	\end{equation}
satisfying the conditions from Theorem \ref{thm:Bog}.

Now we show that $M$ is effective. Since $\Pic(S) = \Z$, we can write $L = \str_S (\lambda)$. By \eqref{SES1} and \eqref{SES2},
	$$c_1(\E) = c_1(L) + c_1(M) = c_1(\alpha H).$$
Thus, $M$ and $\alpha H-L$ are linearly equivalent. By the instability of $\E$, we have 
	$$(2L - \alpha H)\cdot H = (2\lambda  - \alpha) d_S > 0.$$
So 
	\begin{equation}\label{lambda}
2\lambda > \alpha \geq 4.
	\end{equation}
In particular, $\lambda$ is positive.
Thus, the composite map 
	$$ L \to \E \to \mathcal{I}_{\Gamma, S} \otimes \str_S(\alpha)$$ is nonzero, as otherwise $L$ would map to the kernel, $\str_S$, of the right-hand map.
Twisting down by $\lambda$, we obtain a nonzero map $$\str_S \to  \mathcal{I}_{\Gamma, S} \otimes \str_S(\alpha -\lambda) = \mathcal{I}_{\Gamma, S} \otimes M.$$ This implies that
	$$h^0(M) \geq h^0(\mathcal{I}_{\Gamma, S} \otimes M) > 0.$$
Therefore, there is an effective curve $$C_0 \in |M|$$ which contains $\Gamma$. Also, $\alpha > \lambda$ (since $\Gamma$ is nonempty).

Now we approximate the intersection pairing $L\cdot M$. 
Let $d_Z$ denote the length of $Z$. Then by \eqref{SES1} and \eqref{SES2}, we obtain 
$$d_\Gamma = c_2(\E) = L\cdot M + d_Z.$$
Thus $d_\Gamma \geq L\cdot M.$

Collecting inequalities, we have 
$$\deg C = \alpha d_S > d_{\Gamma} \geq L\cdot M = \lambda(\alpha - \lambda) d_S.$$
Combining this inequality with \eqref{lambda}, we get $0< \alpha - \lambda <2$. Thus, $\lambda = \alpha-1$, which proves (2). For (1), notice 
	\begin{align*}
L &= \str_S(\alpha-1)\textrm{, and}\\
M &= \str_S(1).
	\end{align*}
In particular, $\Gamma \subset C_0$ lies in a hyperplane, as desired.

\end{proof}

Now that we know a general divisor in $|f^*\str_{\P^r}(1)|$ will lie in a hyperplane, to prove Theorem \ref{thmc}, it only remains to show that the corresponding pencil of hyperplanes forms a linear pencil and that a member of the pencil contains only one fiber. First we will prove this in the case when $r =1$, and then reduce the general case to this one.

\begin{lem}
\label{thmforP1}
Let $f:C \to \P^1$ be a morphism with $\deg f < \deg C$ (again, with $C$ satisfying the assumptions of Theorem \ref{thmc}). Then $f$ is the projection from an $(n-2)$-plane in $\P^n$.
\end{lem}

\begin{proof}
Let $d_C = \deg(C) $ and $d_S = \deg(S)$.
Let $\Gamma \subset C$ be a fiber of $f$, and suppose it does not span a hyperplane. That is, assume $\Gamma \subset G$, where $G \subset \P^n$ is an $(n-2)$-dimensional linear space. Note that $$(\alpha -1)d_S \leq \gon(C) \leq \deg f$$ by Lemma \ref{lem:key lemma}.
Then projection from $G$ determines a morphism $C \to \P^1$ of degree at most $$d_C - \deg f \leq \alpha d_S -  (\alpha-1)d_S= d_S < (\alpha - 1) d_S \leq \gon(C),$$ which is a contradiction, since the degree cannot be smaller than the gonality. Thus, each fiber  of $f$ spans a hyperplane.

For the sake of contradiction, suppose two distinct fibers lie in $H \subset \P^n$, a hyperplane. Then by Lemma \ref{lem:key lemma}, $$ 2d_S(\alpha-1) \leq 2\deg f < d_C=  d_S \alpha. $$ But $\alpha \geq 4$. Thus, no two fibers span the same hyperplane.

Let $\Gamma_1$ and $\Gamma_2$ be fibers of $f$ contained in hyperplane sections $h_1$ and $h_2$, respectively. There is a linear equivalence between $h_1-\Gamma_1$ and $h_2 -\Gamma_2$. If $h_1 - \Gamma_1$ is not equal to $h_2 - \Gamma_2$ on the level of cycles, then, possibly removing base points, there exists a base-point free pencil on $C$ of degree $$\length(h_1-\Gamma_1) = d_C-\deg f \leq d_S,$$ but, by Lemma \ref{lem:key lemma}, such a pencil cannot exist.

Therefore, $h_1 - \Gamma_1 = h_2 - \Gamma_2$. Thus, since $\Gamma_1$ and $\Gamma_2$ are disjoint, $h_1 - \Gamma_1 = h_1 \cap h_2$. Since we chose $\Gamma_1$ and $\Gamma_2$ arbitrarily, these equalities hold for every pair of fibers.

Let $\G(n-1,n)$ be the Grassmannian of hyperplanes in $\P^n$, and let $\Lambda \subset \G(n-1,n)$ be the pencil of hyperplanes that are spanned by fibers of $f$. Set
$$K = \bigcap_{H \in \Lambda} H \subset \P^n.$$ Assume $\dim K < n-2.$ Then for each pair $H_1, H_2 \in \Lambda$, $$H_1 \cap H_2 \cap C = h_1 \cap h_2.$$ Thus, $\{H_1 \cap H_2 | H_1, H_2 \in \Lambda \}$  is a nonconstant family of $n-2$-planes, whose union in $\P^n$ has dimension at least $n-1$, and intersects $C$ in $h_1 \cap h_2$. But $\length (h_1 \cap h_2) < \length h_1 = d_C$, whereas every hypersurface intersects $C$ in at least $d_C$ points, a contradiction.

Thus, $\dim K = n - 2$, and $\Lambda$ is the unique linear pencil corresponding to the projection from $K$.

 \end{proof}
 
 \begin{proof}[Proof of Theorem \ref{thmc}]
By similar reasoning as in the beginning of the proof of Lemma \ref{thmforP1}, each divisor in the linear system $W  \subseteq H^0 (f^*\str_{\P^r}(1))$ corresponding to $f$ spans a hyperplane, and no two divisors lie in the same hyperplane. That is, there is a natural injection $$\pi: \P(W^\vee) \to \P(V^\vee),$$ where $V = H^0(\str_{\P^n}(1))$. 

It remains to show that the image of $\pi$ is a linear subvariety. Note that it suffices to show that the image of a general line $\P(W^\vee)$ is a line in $\P(V^\vee)$. A general line in $\P(W^\vee)$ is a linear subsystem of dimension one, corresponding to the composition
\begin{equation*}
\xymatrix@=20pt{ C \ar[r]^f & \Im f \ar[r]^{\text{pr}} & \P^1
}
\end{equation*} 
where pr is a projection from an $(r-2)$-plane not meeting $\Im f$. Thus the degree of $\text{pr} \circ f$ is
$$\deg[(\text{pr} \circ f)^* \str_{\P^1}(1)]  = \deg[f^*(\text{pr}^* \str_{\P^1}(1))] = \deg(f^*  \str_{\P^r}(1)).$$
So $\text{pr} \circ f$ satisfies the hypotheses of Lemma \ref{thmforP1}, which guarantees that the image of the linear system associated to $\text{pr} \circ f$ under $\pi$ is a line, as needed.
 \end{proof}
 
Corollary \ref{cord} follows quickly from Lemma \ref{lem:key lemma} (2) and a theorem of Coppens and Martens.

\begin{proof}[Proof of Corollary \ref{cord}]
The lower bound is immediate from Lemma \ref{lem:key lemma} (2).

For the upper bound, since $S$ is non-degenerate, $\deg(S) \geq n-1$. Thus, $\deg(C) \geq 4n-4$.
\cite[Theorem A]{CM} implies that if $\deg(C) \geq 4n-7$, then $C$ has a $(2n-3)$-secant $(n-2)$-plane. Projecting from such a plane yields the upper bound.
\end{proof}

Theorem \ref{thma} follows easily from Theorem \ref{thmc} and Theorem \ref{thm:main theorem}.

\begin{proof}[Proof of Theorem \ref{thma}]
If $n = 2$, the Theorem follows from Noether's Theorem on the gonality of plane curves, so we assume $n \geq 3$. Consider the linear subsystem $\mathcal{D}$ of $|\str_{\P^n}(a_{n-1})|$ consisting of sections vanishing on $C$. Since $C$ is a complete intersection, it is generated in its highest degree $a_{n-1}$, so the base locus of $\mathcal{D}$ is $C$. Thus, by the strong Bertini Theorem (see e.g. \cite[Proposition 5.6]{EH}), a general member of $\mathcal{D}$ is smooth. Choosing such a hypersurface, and proceeding by induction, we can find a smooth complete intersection threefold of type $(a_3, a_4, \ldots, a_{n-1})$ containing $C$ that is smooth. By Theorem \ref{thm:main theorem}, we can then find a smooth complete intersection surface $S$ of type $(a_2, \ldots, a_{n-1})$ satisfying the hypotheses of Theorem \ref{thmc}. (See Remark \ref{assumption remark} for the $n=3$ case.) The conclusion follows by applying Theorem \ref{thmc}.
\end{proof}

We conclude with a restatement of Theorem \ref{thma} for very general curves, which follows from Theorem \ref{thmc} and the classical Noether-Lefschetz theorem.
\begin{thm}\label{thma1}
	Let $C \subset \P^{n}$ be a very general complete intersection curve of type $(a_1, \dots, a_{n-1 })$, with $$2 \leq a_1 \leq a_2 \leq \cdots \leq a_{n-1} , \quad \sum_{i=1}^{n-2} a_i\geq n+1, \quad \text{and}\quad 4\leq a_{n-1}.$$ Then for $r < n$, any morphism $f: C \to \P^r$ satisfying $$\deg f^* \str_{\P^r}(1) < \deg C$$ is obtained by projecting from an $(n-r-1)$-plane. Thus $\gon(C) = a_1 a_2 \cdots a_{n-1} - \gamma$, where $\gamma$ is the maximum number of points on $C$ contained in an $(n-2)$-plane.
\end{thm}

\begin{proof}
Again, we assume $n \geq 3$.
By \cite[Expos\'e XIX (1.2.1)]{SGA}, if $\sum_{i=1}^{n-2}a_i\geq n+1$, the very general complete intersection surface $S$ of type $(a_1,a_2, \cdots,a_{n-2})$ satisfies $\Pic(S)\cong \mathbb{Z}\cdot [\OO_S(1)]$.
We then conclude by applying Theorem \ref{thmc}.
\end{proof}


\section{$(n-2)$-planes secant to general complete intersection curves}
\label{n-2}
In this section, we study the behavior of secant $(n-2)$-planes to complete intersection curves in $\PP^n$ and prove Theorems \ref{thmb} and \ref{corb}. The secancy behavior to curves is a classical subject that has been studied by many people (see for example, \cite{Castelnuovo, CM, ELMS}). Let $n \geq 3$  and $2 \leq a_1 \leq \dots \leq a_{n-1}$ be integers.  In Theorem \ref{thm-existence}, we show that there exists an $(n-2)$-plane which is at least $(2n-2)$-secant to a complete intersection curve of type $(a_1, \dots, a_{n-1})$, unless one of the following conditions holds 
\begin{enumerate}
\item $n=3$ and $(a_1, a_2) = (2,2), (2,3)$ or $(3,3)$
\item $n=4$ and $(a_1, a_2, a_3) = (2,2,2)$. 
\end{enumerate} 
Moreover, we show that if $n \geq 4$ and $n-1 \leq a_1$, then the general complete intersection curve of type $(a_1, \dots, a_{n-1})$ does not have any $(n-2)$-planes which are $(2n-1)$-secant (see Theorem \ref{thm-nonexistence}). When $n=3$, it is well-known that a  general complete intersection with $a_1 \geq 4$ does not have a 5-secant line \cite[Corollary 2.8]{EF}. 
By Theorem \ref{thma}, the gonality of a complete intersection curve is given by projection from a linear $\mathbb{P}^{n-2}$. Consequently, we shall show that when $n-1 \leq a_1$, then the gonality of a general complete intersection of type $(a_1, \dots, a_{n-1})$ is $\prod_{i=1}^{n-1} a_i - 2n +2$.
\subsection{Preliminaries for Theorem \ref{corb}}

In this subsection, we recall the well-known fact that zero-dimensional schemes of length at most $2d+1$ impose independent condition on hypersurfaces of degree $d$ unless they contain a collinear subscheme of length at least $d+2$.

Let $Z$ be a zero-dimensional scheme in $\mathbb{P}^n$ of length $\ell(Z)$. We say that $Z$ {\em imposes independent conditions} on hypersurfaces of degree $d$ if $$h^0(\mathcal{I}_Z(d)) = h^0(\mathcal{O}_{\mathbb{P}^n}(d)) - \ell(Z)$$ or equivalently if $h^1(\mathcal{I}_Z(d))=0$. The following proposition characterizes when $Z$ can fail to impose independent conditions if $d$ is large relative to $\ell(Z)$.

\begin{prop}\label{propConditions}
Let $d$ be a positive integer and let $Z\subset \mathbb{P}^n$ be a zero-dimensional scheme of length at most $2d+1$. Assume that the maximal length collinear subscheme $W$ of $Z$ has length $k$. Then 
\[
h^{1}(\mathbb{P}^{n},\mathcal{I}_Z(d))=\max\{0,k-d-1\}.
\]
In particular, $Z$ imposes independent conditions on hypersurfaces of degree $d$ if and only if $k \leq d+1$. 
\end{prop}

\begin{proof}
The following lemma allows us to do induction on $d$ and $n$.

\begin{lem}\label{lemma: residue}
Let $Z\subset \mathbb{P}^n$, $n\geq 1$, be a zero-dimensional subscheme
and $H\subset\mathbb{P}^n$ be a hyperplane.
Then we have the following commutative diagram
\[
\xymatrix{
& 0\ar[d] & 0\ar[d] & 0\ar[d] 
\\
0\ar[r]&\mathcal{I}_{R}(-1)\ar[r]\ar[d]&\mathcal{I}_Z\ar[r]\ar[d]&\mathcal{I}_{Z\cap H/H}\ar[r]\ar[d]&0\\
0\ar[r]&\mathcal{O}_{\mathbb{P}^n}(-1)\ar[r]\ar[d]&\mathcal{O}_{\mathbb{P}^n}\ar[r]\ar[d]&\mathcal{O}_{H}\ar[r]\ar[d]&0\\
0\ar[r]&\mathcal{O}_{R}(-1)\ar[r]\ar[d] &\mathcal{O}_{Z}\ar[r]\ar[d] &\mathcal{O}_{Z\cap H/H}\ar[r]\ar[d] &0\\
&0&0&0
}
\]
where $\mathcal{I}_{Z\cap H/H}$ is the ideal sheaf $Z\cap H$ in $H$ and $R$ is the residual subscheme of $Z\cap H$ in $Z$.
Moreover, $$\ell(Z)=\ell(R)+\ell(Z\cap H).$$
\end{lem}
\begin{proof}
The middle column is the standard ideal sheaf exact sequence.  Restricting the middle column to the hyperplane $H$, we obtain the right column of the diagram. We may extend the morphism $\mathcal{O}_Z\rightarrow\mathcal{O}_{Z\cap H/H}$ to the following exact sequence:
\[
0\rightarrow\ker\rightarrow\mathcal{O}_Z(-1)\xrightarrow{\cdot H}\mathcal{O}_Z\rightarrow\mathcal{O}_{Z\cap H/H}\rightarrow 0.
\]
We now define the subscheme $R\subset Z$ to be cut out by the ideal sheaf $\ker\subset \mathcal{O}_Z(-1)\cong \mathcal{O}_Z$, where the isomorphism follows from the fact that $Z$ is zero-dimensional.
Now one can easily construct the left column of the diagram. Since
\[
0\rightarrow H^{0}(\mathbb{P}^n,\mathcal{O}_{R}(-1))\rightarrow H^{0}(\mathbb{P}^n,\mathcal{O}_{Z})\rightarrow H^{0}(H,\mathcal{O}_{Z\cap H/H})\rightarrow 0
\]
is exact, we have
$\ell(Z)=\ell(R)+\ell(Z\cap H)$.
\end{proof}

We can now complete the proof of the proposition.  The statement is clearly true when $n=1$ or $d=1$. Let $W$ be a maximal collinear subscheme of $Z$ and let $\lambda$ be the line spanned by $W$.  Take a general hyperplane $H$ containing $\lambda$. By Lemma \ref{lemma: residue}, we have the following short exact sequence:
\[
0\rightarrow\mathcal{I}_R(d-1)\rightarrow\mathcal{I}_Z(d)\rightarrow\mathcal{I}_{Z \cap H/H}(d)\rightarrow 0,
\]
where $R$ denotes the residual subscheme of $W=Z\cap H$ in $Z$. Note that $\ell(W)=k \geq 2$, hence $\ell(R) \leq \ell(Z) - k \leq 2(d-1) + 1$.
We have $$h^{2}(\mathcal{I}_R(d-1))=h^{2}(\mathcal{O}_{\mathbb{P}^n}(d-1))=0.$$
Let $k'$ be the length of a maximal collinear subscheme of $R$. Note that $k' \leq k$. By induction on $n$ and $d$, we have
$$h^1(\mathcal{I}_R(d-1)) = \max\{0, k'-d\} \quad  \mbox{and} \quad h^{1}(\mathcal{I}_{W/H}(d))=\max\{0,k-d-1\}.$$
Therefore, if $k'\leq k\leq d$, then $$h^{1}(\mathcal{I}_R(d-1))= h^{1}(\mathcal{I}_{Z \cap H/H}(d))= h^{1}(\mathcal{I}_Z(d))=0.$$
On the other hand, if $k\geq d+1$, then $$k' \leq \ell(R)\leq 2d+1-(d+1)=d,  \quad \mbox{and hence} \quad h^{1}(\mathcal{I}_R(d-1))=0.$$
Thus, 
$$h^{1}(\mathcal{I}_Z(d))=h^{1}(\mathcal{I}_{Z \cap H/H}(d))=\max\{0,k-d-1\}.$$
\end{proof}

\subsection{Existence of $(2n-2)$-secant $(n-2)$-planes}

In this subsection, we shall prove Theorem \ref{thmb}, that a complete intersection curve in $\mathbb{P}^n$ of sufficiently large degree has a $(2n-2)$-secant $(n-2)$-plane. Our main tool is Castelnuovo's enumerative formula.

\begin{thm}\label{thm-existence}
Let $n\geq 3$ and $2\leq a_1\leq a_2 \leq \cdots \leq a_{n-1}$ be integers. 
Let $C\subset \mathbb{P}^n$ be a complete intersection curve of type $(a_1,\cdots,a_{n-1})$.
Then $C$ has an $(n-2)$-plane which is at least  $(2n-2)$-secant  to $C$ unless
$$n=3 \ \mbox{and} \ (a_1,a_2)=(2,2), (2,3), (3,3); \ \mbox{or} \ n=4 \  \mbox{and} \ (a_1,a_2,a_3)=(2,2,2).$$
\end{thm}
\begin{proof}
Castelnuovo computed the class of the locus of $(2n-2)$-secant $(n-2)$-planes to a curve $C$ in $\mathbb{P}^n$ \cite{Castelnuovo}.  A modern proof of a more general formula due to Macdonald can be found in \cite[VIII, Proposition 4.2]{ACGH} and the equivalence of the two formulae in this case is explained in \cite[Section 1]{ELMS}. The expected dimension of this locus is zero and the expected number of $(2n-2)$-secant $(n-2)$-planes is given by 
\[
C(d,g,n)=\sum_{i=0}^{n-1}\frac{(-1)^{i}}{n-i}{d-n-i+1\choose n-1-i}{d-n-i\choose n-1-i}{g\choose i},
\]
where $d$ and $g$ are the degree and the genus of the curve, respectively. To prove the theorem, we need to show that $C(d,g,n)$ is nonzero for complete intersection curves except for the degrees and $n$ specified in the theorem. 

\begin{prop}
Let $n\geq 3$ and $2\leq a_1\leq a_2 \leq \cdots \leq a_{n-1}$ be integers.  For a complete intersection curve of type $(a_1, \cdots, a_{n-1})$ in $\mathbb{P}^n$, the number $C(d,g,n)$ satisfies the following properties:
\begin{center}
\begin{tabular}{ c |c c c}
        & $C(d,g,n)=0$ & $C(d,g,n)<0$ & $C(d,g,n)>0$ \\ 
        \hline
$n=3$ & $(a_1,a_2)=(2,2),$ or $(2,3)$ or $(3,3)$ & $(a_1,a_2)=(2,4)$ & Otherwise \\      
 $n=4$ & $a_1=a_2=a_3=2$ & $a_1=a_2=2$ and $3\leq a_3\leq11$ & Otherwise\\  
 $n=5$ & Never & $a_1=a_2=a_3=a_4=2$ & Otherwise\\
 $n\geq6$ & Never & Never & Always
\end{tabular}
\end{center}
\end{prop}

\begin{proof}
The case where $n=3$, $d\geq 10$ is covered in \cite[Proposition 3.1]{HS}. The remaining cases when $n=3$ can easily be computed by hand.
Therefore, we now assume $n\geq 4$. We compare consecutive terms in the alternating sum for $C(d,g,n)$.
Suppose $0\leq i\leq n-2$ is even.
If we can show that the sum of the $i$-th and the $(i+1)$-th term is positive, for all $i$ even and $0\leq i\leq n-2$, then $C(d,g,n)>0$.

The sum of the $i$-th and the $(i+1)$-th term equals
$$A_i \cdot \left((i+1)(d-n-i+1)(d-n-i) -(n-i)(n-i-1)(g-i)\right),$$ where $A_i$ is the positive constant given by the following formula
$$A_i = \frac{g!}{(i+1)!(g-i)!}\cdot \frac{(d-n-i)!}{(n-i)!(d-2n+2)!}\cdot \frac{(d-n-i-1)!}{(n-i-1)!(d-2n+1)!}.$$
Hence, it suffices to determine when
\[
f(i)=(i+1)(d-n-i+1)(d-n-i)
-
(n-i)(n-i-1)(g-i)
\]
is positive.
We compute the derivative of $f$ with respect to $i$ for $0\leq i\leq n-2$.
\begin{align}
f'(i)&=6i^2+(-4d-2g+2)\cdot i+[(d-n)^2-(d-n)-1+2ng-g+n(n-1)]\nonumber\\
&\geq (-4d-2g+2)(n-2)+[(d-n)^2-(d-n)-1+2ng-g+n(n-1)]\nonumber\\
&=d^2+(-6n+7)d+3g+2n^2+2n-5\label{eqn:derivative}
\end{align}
For $n\geq 6$, using the fact that the genus of the complete intersection curve is 
\begin{equation}\label{eqn:genus}
g=\frac{d\left((\sum_{i=1}^{n-1} a_i)-n-1\right)}{2}+1,
\end{equation}
the term \ref{eqn:derivative} is clearly positive as long as $a_j\geq 2$ for all $1\leq j\leq n-1$.  We conclude that $f'(i)>0$ and therefore $f$ is increasing in $i$.
Similarly, for $4\leq n \leq 5$ if we additionally assume that $a_{n-1}\geq 3$, then $f'(i)>0$ and therefore $f$ is increasing in $i$.

Hence, to prove that $C(d,g,n)$ is positive for given values of $a_i$ and $n$, it suffices to show that $f(0)\geq 0$. Substituting \ref{eqn:genus} into  $f(0)=(d-n+1)(d-n)-n(n-1)g$,
we obtain
\[
f(0)=d\left(d-2n+1-\frac{n(n-1)(\sum_{i=1}^{n-1}a_i-n-1)}{2}\right).
\]
Since $d$ is positive, it suffices to analyze the term  $$h(d,n)= \left(d-2n+1-\frac{n(n-1)(\sum_{i=1}^{n-1}a_i-n-1)}{2}\right).$$ 
Differentiating $h(d,n)$ with respect to $a_i$,
we see that $$\frac{\partial h(d,n)}{\partial a_i} =  \frac{d}{a_i}-\frac{n(n-1)}{2}\geq 2^{n-2}-\frac{n(n-1)}{2}.$$  The latter quantity is positive if $n\geq 6$.
Consequently, for $n\geq 6$, $f(0)$ is increasing with respect to $a_i$. Now there are several cases.
\begin{enumerate}
\item If $n\geq 9$ and $a_1 = \cdots = a_{n-1}=2$, one sees that $f(0)$ is positive.  Therefore, when $n \geq 9$,  $C(d,g,n)$ is always positive for complete intersection curves.
\item If  $n = 8$ and $a_7 \geq 3$; or $n =7$ and $a_6 \geq 6$; or  $n=6$ and $a_5 \geq 27$, then $f(0)$ is positive. Therefore, when $6 \leq n \leq 8$, $C(d,g,n)$ is positive for complete intersection curves except in a finite number of cases. In these cases, a simple computer check easily shows that $C(d,g,n)$ is always positive.
\item If $n=4$ or $5$, assume that $5 \leq a_{n-2} \leq a_{n-1}$. Then 
$$\frac{\partial h(d,n)}{\partial a_i} =  \frac{d}{a_i}-\frac{n(n-1)}{2}\geq 2^{n-3} \times 5-\frac{n(n-1)}{2} > 0.$$ Hence, $f$ is increasing in the $a_i$. Moreover, it is easy to check that $f(0) > 0$ when $n=4, a_1=2, a_2 = a_3 =5$ and $n=5, a_1=a_2 = 2, a_3 = a_4 = 5$. We conclude that $C(d,g,n) > 0$ for complete intersection curves in $\mathbb{P}^4$ and $\mathbb{P}^5$ as long as $a_{n-2} \geq 5$. 

For each of the finitely many choices, $2 \leq a_1 \leq \dots \leq a_{n-2} \leq 4$, $C(d,g,n)$ is a polynomial in $a_{n-1}$ with positive leading coefficient. A computer can easily compute the largest root of this polynomial in each of the cases. When $n=5$ and $a_1 = a_2 = a_3 =2$, the largest root is between $2$ and $3$. Moreover, $C(d,g,n) < 0$ when $n=5$ and $a_4 =2$. In all other cases, the largest root is less than $a_3$, hence $C(d,g,n)$ is positive for all remaining cases. Similarly, when $n=4$ and $a_1 = a_2 = 2$, the largest root is between $11$ and $12$ and for $3 \leq a_3 \leq 11$, $C(d,g,n) < 0$. Moreover, when $a_3 =2$, $C(d,g,n) =0$. In all other cases, the largest root is less than $a_2$ and $C(d,g,n) >0$.
\end{enumerate}
This analysis completes the proof of the proposition.
\end{proof}
To conclude the proof of the theorem, we simply observe that since Castelnuovo's formula computes the class of the locus of $(2n-2)$-secant $(n-2)$-planes, this locus cannot be empty if the class is nonzero. 
\end{proof}

The cases when $C(d,g,n)$ is negative or zero have clear geometric explanations. 

\begin{exmp}
Let $C$ be a complete intersection of type $(2,a_2)$ in $\PP^3$. Let $Q$ be a quadric surface containing $C$. Then any line which is at least trisecant  to $C$ must be contained in $Q$. Conversely, any of the lines in $Q$ is $a_2$-secant to $C$. We conclude that in this case, $C$ does not have any $4$-secant lines for $a_2 \leq 3$. If $a_2 \geq 4$, $C$ has a one-parameter family of $4$-secant lines. In fact, these lines are $a_2$-secant to $C$. 
\end{exmp}

\begin{exmp}
Let $C$ be a complete intersection of type $(3,a_2)$ in $\PP^3$. Let $S$ be a cubic surface containing $C$. Then any $4$-secant line must be contained in $S$. Conversely, any line on $S$ is $a_2$-secant to $C$. Hence, when $a_2=3$, $C$ does not have any $4$-secant lines. Otherwise, assuming $S$ is smooth, $C$ has 27 $a_2$-secant lines. 
\end{exmp}

\begin{exmp}
Let $C$ be a complete intersection of type $(2,2,a_3)$ in $\PP^4$. Then the pencil of quadric threefolds containing $C$ has a singular member $Q$. The quadric $Q$ has a one-parameter family of planes. Hence, $C$ has a one-parameter family of $2a_3$-secant planes. 
Moreover, when $a_3=2$, $C$ has no $5$-secant planes. Indeed, any such $5$-secant plane must be contained in a member of the web of quadrics, but then the curve must intersect the plane at a length-$4$ subscheme. A contradiction.
\end{exmp}

\begin{exmp}
Let $C$ be a complete intersection of type $(2,2,2,2)$ in $\PP^5$. Then the web of quadrics containing $C$ contains a quadric of corank 2. Such a quadric has a one-parameter family of 3-planes. Consequently, $C$ has a one-parameter family of $8$-secant 3-planes.
\end{exmp}

\subsection{Non-existence of $(2n-1)$-secant $(n-2)$-planes}

In this subsection, we shall prove Theorem \ref{corb}, that when the degrees of the defining equations are sufficiently large, then the general complete intersection curve does not contain a $(2n-1)$-secant $(n-2)$-plane and has exactly $C(d,g,n)$-many $(2n-2)$-secant $(n-2)$-planes, all of them intersecting $C$ transversely.

\begin{thm}\label{thm-nonexistence}
Let $n\geq 4$,
$n-1\leq a_1\leq a_2\leq \cdots\leq a_{n-1}$.
Let $C\subset \mathbb{P}^{n}$ be the general complete intersection curve of type $(a_1,\cdots,a_{n-1})$.
Then $C$ does not have a $(2n-1)$-secant $(n-2)$-plane.
\end{thm}
\begin{proof}
Let $\Gamma$ denote the incidence correspondence parameterizing tuples $$(Z,\Lambda,X_{a_1},X_{a_2},\cdots,X_{a_{n-1}}),$$ where $\Lambda$ is an $(n-2)$ plane, $Z$ is a curvilinear zero-dimensional scheme of $\Lambda$ of length $2n-1$ and $X_{a_i}$ are hypersurfaces of degree $a_i$ containing $Z$. 
Let $\Gamma_k$ be the locally closed subscheme of $\Gamma$ defined by the condition that the maximal length among all collinear subschemes of $Z$ is $k$. 
To prove the theorem it suffices to show that the natural projection $\pi$  from $\Gamma$ to $\prod_{i=1}^{n-1} | \OO_{\PP^n}(a_i)|$ is not dominant, or each $\Gamma_k$ does not dominate.

Let $\pi_k$ denote the restriction of $\pi$ to $\Gamma_k$. The next lemma shows that if $k > a_2$, then $\pi_k$ cannot be dominant.
\begin{lem}\label{lemma:nolines}
Let $X_1$ and $X_2$ be general hypersurfaces in $\mathbb{P}^{n}$ of degrees $a_1$ and $a_2$ respectively, with $a_1+a_2>2n-4$. Then $X_1\cap X_2$ does not contain a line.
\end{lem}
\begin{proof}
Let $G$ be the Grassmannian of lines in $\PP^{n}$ and let $\tau$ be the tautological rank $2$ vector bundle on $G$.
There is a canonical isomorphism
$$H^{0}(\mathbb{P}^{n},\mathcal{O}(a_1)\oplus\mathcal{O}(a_2))\cong H^{0}(G,\Sym^{a_1}\tau^{*}\oplus \Sym^{a_2}\tau^{*}),$$
such that the zero locus of a global section of 
$\Sym^{a_1}\tau^{*}\oplus \Sym^{a_2}\tau^{*}$
in $G$ consists of the lines contained in the corresponding complete intersection \cite[p.26-27]{Bor}.
Since 
$\Sym^{a_1}\tau^{*}\oplus \Sym^{a_2}\tau^{*}$
is generated by global sections and has rank 
$(a_1+1)+(a_2+1)>2(n-1)=\dim G$,
a general section does not vanish.
\end{proof}

If 
$k > a_2$,
then $X_1$ and $X_2$ both contain the line spanned by the maximal collinear subscheme. Hence, by Lemma \ref{lemma:nolines}, $\pi_k$ cannot dominate.
We may assume that 
$k \leq a_2$. 

The space $\Theta$ of pairs $(Z, \Lambda)$, where $\Lambda$ is an $(n-2)$-plane and $Z$ is a zero-dimensional curvilinear subscheme of $\Lambda$ of length $2n-1$ is irreducible of dimension $$\dim(\Theta) = 2(n-1) + (2n-1)(n-2) = n (2n-3).$$ Let $\Theta_k$ denote the locus where the maximal length of collinear subschemes of $Z$ is $k$. Let $W$ be the collinear length $k$ subscheme. By Lemma \ref{lemma: residue}, there is a well-defined residual scheme $R\subset Z$ of $W$, which is also curvilinear of length $2n-1-k$. The dimension of the space of collinear subschemes in $\Lambda$ of length $k$ is $2(n-3) + k$. Consequently, the dimension of $\Theta_k$ is bounded by 
$$
\dim(\Theta_k) \leq 2(n-3) + k + (2n-1-k)(n-2) + 2(n-1) = n(2n-3) -k (n-3) + 2n-6.
$$
Observe that $k \geq 2$. If $k \leq a_1 +1$, then by Proposition \ref{propConditions}, $Z$ imposes independent conditions on $|\OO_{\PP^n}(a_i)|$ for $i \geq 1$. Hence, the fiber dimension $\Gamma_k$ over $\Theta_k$ is $$\sum_{i=1}^{n-1} \dim |\OO_{\PP^n}(a_i)| - (2n-1)(n-1).$$ 
Thus, 
\begin{align*}
    \dim \Gamma_k&\leq n(2n-3) -k (n-3) + 2n-6+\sum_{i=1}^{n-1} \dim |\OO_{\PP^n}(a_i)| - (2n-1)(n-1)\\
    &=\sum_{i=1}^{n-1} \dim |\OO_{\PP^n}(a_i)|-(k-2)(n-3)-1<\sum_{i=1}^{n-1} \dim |\OO_{\PP^n}(a_i)|
\end{align*}
and $\pi_k$ cannot dominate $\prod_{i=1}^{n-1}|\OO_{\PP^n}(a_i)|$.

If 
$a_2 \geq k > a_1 +1$,
then by Proposition \ref{propConditions}, $Z \in \Gamma_k$ imposes independent conditions on 
$|\OO_{\PP^n}(a_i)|$ for $i \geq 2$
and fails to impose independent conditions on 
$|\OO_{\PP^n}(a_1)|$ by $k-a_1-1$.
Hence, the fiber dimension of $\Gamma_k$ over $\Theta_k$ is 
$$\sum_{i=1}^{n-1} \dim | \OO_{\PP^n}(a_i)| - (2n-1)(n-1) + k -a_1 -1.$$
Therefore,
\begin{align*}
\dim\Gamma_k&\leq n(2n-3) -k (n-3) + 2n-6+\sum_{i=1}^{n-1} \dim | \OO_{\PP^n}(a_i)| - (2n-1)(n-1) + k -a_1 -1\\
&=\sum_{i=1}^{n-1} \dim | \OO_{\PP^n}(a_i)|-(k-a_1-1)(n-4)-(a_1-1)(n-3)-1<\sum_{i=1}^{n-1} \dim | \OO_{\PP^n}(a_i)|,
\end{align*}
and $\pi_k$ cannot be dominant. We conclude that the general $C$ does not contain a $(2n-1)$-secant $(n-2)$-plane.
\end{proof}
Let us recall the definition of the Clifford index.
\begin{defn}
The Clifford index of a curve $C$ is defined by:
$$
\Cliff(C):=\min\{\Cliff(\mathcal{L})\mid \mathcal{L}\in\Pic(C), h^0(\mathcal{L})\geq 2, h^1(\mathcal{L})\geq 2\}
$$
where $\Cliff(\mathcal{L}):=\deg(\mathcal{L})-2h^0(\mathcal{L})+2$.
\end{defn}

By Theorem \ref{thma1}, the gonality of a very general complete intersection curve of type $(a_1,a_2,\cdots,a_{n-1})$ where $\sum_{i=1}^{n-2} a_i \geq n+1$ and $a_{n-1}\geq 4$ is computed by linear projection from an $(n-2)$-plane.
Consequently, we deduce the following theorem, which implies Theorem \ref{corb}.
\begin{thm}\label{thmmain}
Let $n\geq 4$. Suppose either
\begin{enumerate}
    \item $n-1 \leq a_1 \leq \cdots \leq a_{n-1}$, $a_{n-1}\geq 4$ and $C$ be a \textbf{very general} complete intersection curve of type $(a_1, \dots, a_{n-1})$.
    \item $\max\{n-1,4\}\leq a_1<a_2\leq\cdots\leq  a_{n-1}$ and $C$ be a \textbf{general} complete intersection curve of type $(a_1, \dots, a_{n-1})$.
\end{enumerate}
Then $\gon(C)=\deg(C)-2n+2$ and the gonality of $C$ is computed by projection from $(n-2)$-planes.
Moreover, there are finitely many $(2n-2)$-secant $(n-2)$-planes and
these secant $(n-2)$-planes intersect $C$ at exactly $(2n-2)$ distinct points.
The Clifford index of $C$ equals $\gon(C)-2=\deg(C)-2n$.
\end{thm}

\begin{proof}
Applying Theorem \ref{thma1} and Theorem \ref{thma} respectively, we see that the gonality of $C$ is computed by linear projection in either case. Thus Theorems \ref{thm-existence} and \ref{thm-nonexistence} yield that $\gon(C) = \deg(C) -2n+2$. 

In order to prove the finiteness of the number of $(2n-2)$-secant $(n-2)$-planes, we follow the notation and idea used in the proof of Theorem \ref{thm-nonexistence}.

If 
$k > a_2$,
then $X_1$ and $X_2$ both contain the line spanned by the length $k$ collinear subscheme. Hence, by Lemma \ref{lemma:nolines}, $\pi_k$ cannot dominate.

So we assume hereafter that $k\leq a_2$, and we bound the dimension of the schemes $\Gamma_k$. We note that
$$
\dim(\Theta_k) \leq 2(n-3) + k + (2n-2-k)(n-2) + 2(n-1)=2n^2-2n-kn+3k-4.
$$
For $k\leq a_1+1$, by applying Proposition \ref{propConditions}, we see that the fiber dimension over $\Theta_k$ is
$$
\sum_{i=1}^{n-1} \dim |\OO_{\PP^n}(a_i)| - (2n-2)(n-1).
$$
Thus,
\begin{align*}
    \dim \Gamma_k&\leq 2n^2-2n-kn+3k-4+\sum_{i=1}^{n-1} \dim |\OO_{\PP^n}(a_i)| - (2n-2)(n-1)\\
    &=\sum_{i=1}^{n-1} \dim |\OO_{\PP^n}(a_i)|-(k-2)(n-3)\leq\sum_{i=1}^{n-1} \dim |\OO_{\PP^n}(a_i)|,
\end{align*}
with the last equality holds if and only if $k=2$.

For $a_1+1<k\leq a_2$, again by Proposition \ref{propConditions},
\begin{align*}
\dim\Gamma_k&\leq 2n^2-2n-kn+3k-4+\sum_{i=1}^{n-1} \dim | \OO_{\PP^n}(a_i)| - (2n-2)(n-1) + k -a_1 -1\\
&=\sum_{i=1}^{n-1} \dim | \OO_{\PP^n}(a_i)|-(k-a_1-1)(n-4)-(a_1-1)(n-3)<\sum_{i=1}^{n-1} \dim | \OO_{\PP^n}(a_i)|.
\end{align*}
So $\pi_k$ does not dominate.

We have shown that only $\pi_2$, amongst all $2\leq k\leq 2n-2$, can possibly dominate $\prod_{i=1}^{n-1} | \OO_{\PP^n}(a_i)|$.
Theorem \ref{thm-existence} shows that some irreducible component $\Gamma'$ of $\Gamma$ does dominate $\prod_{i=1}^{n-1} | \OO_{\PP^n}(a_i)|$.
Note that $\Gamma' \subset \Gamma_2$, and any such $\Gamma'$ must have same dimension as $\prod_{i=1}^{n-1} | \OO_{\PP^n}(a_i)|$ and the projection $\Gamma'\to \prod_{i=1}^{n-1} | \OO_{\PP^n}(a_i)|$ must be generically finite.
Moreover, note that the locus in $\overline{\Theta_2}$ where $Z$ is non-reduced forms a proper closed subscheme.
The preimage of this subscheme in $\Gamma$, intersecting with $\Gamma'$ cannot dominate $\prod_{i=1}^{n-1} | \OO_{\PP^n}(a_i)|$.
Thus, for general $C$, there are finitely many $(2n -2)$-secant $(n-2)$-planes, each of them intersecting $C$ at exactly $(2n-2)$ distinct points.
Since there are only finitely many linear systems that compute the gonality of $C$, 
by a theorem of Coppens and Martens \cite[Corollary 2.3.1]{CM},
the Clifford index of $C$, $\Cliff(C)=\gon(C)-2=a_1a_2\cdots a_{n-1}-2n$.
\end{proof}



It is well-known that when $n=3$ and $a_1 \geq 4$, the conclusion of the previous theorem holds \cite{HS}.

Theorem \ref{thm-nonexistence} is likely not optimal. We expect the following statement to hold.

\begin{conj}\label{conj-no2n-1}
Suppose $n\geq 4$, $2\leq a_1\leq a_2\leq \cdots\leq a_{n-1}$,
$\sum_{i=1}^{n-2} a_i\geq n+1$ and $a_{n-1} \geq n-1$.
Let $C\subset \mathbb{P}^n$ be a very general  complete intersection curve of type $(a_1,a_2,\cdots,a_{n-2},a_{n-1})$. Then $C$ has no $(2n-1)$-secant $(n-2)$-plane and its gonality is $\prod_{i=1}^{n-1}a_i - 2n+2$.
\end{conj}

We give a heuristic argument for the conjecture. We fix a very general complete intersection \textbf{surface} $X$ of type $(a_1, \dots, a_{n-2})$. Since $\sum_{i=1}^{n-2} a_i\geq n+1$, by the Noether-Lefschetz Theorem, $\Pic(X) = \ZZ \cdot [\OO_X(1)]$. 
Let $P$ be an $(n-2)$-plane. Then $P\cap X$ is a $0$-dimensional subscheme. Otherwise, $P\cap X$ would contain an irreducible curve $B$. Notice that $\deg(B) < \deg (X)$. Indeed, if $\deg(B) \geq \deg(X)$, we could choose a point $p$ not in $P$, and we would get a hyperplane section spanned by $B$ and $p$ of degree bigger than $\deg(X)$, contradicting that $X$ is nondegenerate. However, having $\deg (B) < \deg(X)$ would contradict the Noether-Lefschetz Theorem. We conclude that every $(n-2)$-plane intersects $X$ in a zero-dimensional scheme. 

Consider the incidence correspondence $\Gamma$ parameterizing triples $(Z, P, X_{a_{n-1}})$, where $P$ is an $(n-2)$-plane and $Z$ is a length $2n-1$ curvilinear subscheme of $X \cap P$ and $X_{a_{n-1}}$ is a hypersurface of degree $a_{n-1}$ containing $Z$. Also, consider the image of the projection of $\Gamma$ to the first two factors, call it $\Theta$. We claim that the projection of $\Theta$ to the second factor (i.e. the Grassmannian of $(n-2)$-planes) is dominant and generically finite. Indeed, for any $(n-2)$-plane $P$, $P\cap X$ is a length $\prod_{i=1}^{n-2}a_i$ zero-dimensional subscheme. If $\prod_{i=1}^{n-2}a_i < 2n-1$, then the incidence correspondence $\Gamma$ has to be empty, and we are done. Thus, we may assume $\prod_{i=1}^{n-2}a_i \geq 2n-1$. For $P$ general, $P\cap X$ is a set of distinct points. Therefore, $P\cap X$ has finitely many subsets $Z$ of size $2n-1$. For each subset $Z$ of $P\cap X$ of size $2n-1$, there is a hypersurface of degree $a_{n-1}$ containing $Z$, since $\dim|\OO_{\PP^n}(a_{n-1})|\geq2n-1$. 

Let $\Theta'$ be one of the irreducible components of $\Theta$ that dominate the Grassmannian via the second projection. It has dimension $2(n-1)$ equal to the dimension of the Grassmannian $\mathbb{G}(n-2, n)$. Let $(Z,P)$ be a general point of $\Theta'$. By Proposition \ref{propConditions}, if the largest collinear subscheme of $Z$ has length at most $a_{n-1} +1$, then $Z$ imposes independent conditions on $|\OO_{\PP^n}(a_{n-1})|$, hence the fiber of $\Gamma$ over $(Z,P)$  has dimension $\dim |\OO_{\PP^n}(a_{n-1})|- 2n+1$. Therefore,  the irreducible component(s) of $\Gamma$ over $\Theta'$ has dimension at most $\dim |\OO_{\PP^n}(a_{n-1})|-1$ and cannot dominate  $|\OO_{\PP^n}(a_{n-1})|$. If $Z$ fails to impose independent conditions on $|\OO_{\PP^n}(a_{n-1})|$, then $Z$ must have a collinear subscheme of length at least $a_{n-1} + 2$. Then the line spanned by this collinear subscheme would be contained in all the hypersurfaces containing $C$, contradicting that $C$ is a smooth, irreducible curve. 

The reason why this is not a rigorous proof is that $\Gamma$ may be reducible and a component dominating $|\OO_{\PP^n}(a_{n-1})|$ may fail to dominate $\Theta$. We do not expect this to happen, but we do not have an argument for this.

More generally, one can conjecture the following.

\begin{conj}\label{conj-correctno}
Let $C$ be a general complete intersection of type $(a_1, \dots, a_{n-1})$ in $\PP^n$. Assume that $\sum_{i=1}^{n-2} a_i \geq n+1$. Then $C$ has finitely many $(2n-2)$-secant $(n-2)$-plane and their number is determined by Castelnuovo's formula. 
\end{conj} 

\begin{conj}\label{conj-transverse}
Assume that $\sum_{i=1}^{n-2} a_i \geq n+1$. Then there exists a complete intersection curve of type  $(a_1, \dots, a_{n-1})$ in $\PP^n$ with finitely many $(2n-1)$-secant $(n-2)$-planes all of which are transverse to $C$.
\end{conj}

We remark that Conjecture \ref{conj-transverse} implies Conjecture \ref{conj-no2n-1}. Namely, in the heuristic argument, the loci where $Z$ is not reduced could not dominate $|\OO_{\PP^n}(a_{n-1})|$. If we restrict our attention to the locus where $Z$ is reduced, the dimension count in the heuristic argument is rigorous.

\section{A Cayley-Bacharach Conjecture}
\label{sec:CB}

By modifying the proof of Lemma \ref{lem:key lemma}, we are able to prove the following special case of one of the Cayley-Bacharach conjectures (see \cite[Conjecture CB12]{EGH}).

\begin{thm}
\label{CBconj}
Let $\Gamma$ be any subscheme of a zero-dimensional complete intersection of hypersurfaces of degrees $d_1 \leq \cdots \leq d_n$ in $\P^n$, so that $\Gamma$ is contained in a complete intersection surface $S$ of type $(d_3, \ldots, d_n)$ with $\Pic(S)$ generated by the hyperplane class. Set $$k := d_3 + \cdots + d_n - n -1.$$
If $\Gamma$ fails to impose independent conditions on hypersurfaces of degree $k + e+2 $, where $ 0 \leq e \leq d_2-1$, then $$\deg(\Gamma) \geq (e+1)\cdot d_3 \cdot d_4 \cdot \cdots \cdot d_n.$$ 
\end{thm}

\begin{rem}
In the notation of the original statement of the conjecture, we are considering the case $s=3$, and setting $m=k+e+2$. Notice that in this case, we are able to obtain a stronger bound than the one given in the conjecture.
\end{rem}

Before the proof, we need a slightly more general formulation of Theorem \ref{thm:GH} in order to deal with non-reduced 0-cycles.  

\begin{thm}[cf. {{{\cite[Proposition 3.9]{Laz}}}}]
\label{genserre}
Let $S$ be a smooth projective surface, $Z \subset S$ a zero-dimensional subscheme, and $L$ a line bundle on $S$. Given an element $\eta \in \Ext^1(L \otimes \mathcal{I}_{Z,S}, \str_S)$, denote by $\mathcal{F}_{\eta}$ the sheaf arising from the extension:
$$0 \to \str_S \to \mathcal{F}_{\eta} \to L \otimes \mathcal{I}_{Z,S} \to 0.$$
Then $\mathcal{F}_{\eta}$ fails to be locally free if and only if there exists a proper (possibly empty) subscheme $Z' \subsetneq Z$ such that $$\eta \in \Im \left\{ \Ext^1(L \otimes \mathcal{I}_{Z',S}, \str_S) \to \Ext^1(L \otimes \mathcal{I}_{Z,S}, \str_S)\right\}.$$
\end{thm}

\begin{proof}[Proof of Theorem \ref{CBconj}]
Set $m := k +e +2$. 

First we show that there is some non-empty subscheme $Z \subseteq \Gamma$ such that $h^1(\mathcal{I}_{Z,\P^n}(m)) \neq 0$ but $h^1(\mathcal{I}_{Z',\P^n}(m)) = 0$ for all proper subschemes $Z' \subsetneq Z$ by induction on the length of $\Gamma$. Notice that since $h^1(\str_{\P^n} (m)) =0$, the condition $h^1(\mathcal{I}_{Z,\P^n}(m)) \neq 0$ is equivalent to $Z$ failing to impose independent conditions on $\str_{\P^n}(m)$, so if the proper subschemes of $\Gamma$ all satisfy $h^1(\mathcal{I}_{Z,\P^n}(m))=0$, we are done. Otherwise, there is some nonempty $Z \subsetneq \Gamma $ such that $h^1(\mathcal{I}_{Z,\P^n}(m)) \neq 0$, and we are done by the induction hypothesis. For the base case, assume $\Gamma$ consists of a single point. Then the only proper subscheme is the empty set, which trivially imposes independent conditions on $\str_{\P^n}(m)$, as desired. Since $\deg(\Gamma) \geq \deg(Z)$, we can replace $\Gamma$ with $Z$ for the remainder of the proof.

Since $S$ is a complete intersection with embedding line bundle $\str_S(1)$, we know $h^1(\str_S(m)) = 0$. Thus, we obtain the following commutative diagram with exact rows:

\begin{equation}
\xymatrix@=20pt{ h^0(\str_{\P^n}(m)) \ar[r] \arsurj[d] & h^0(\str_Z(m)) \ar[r] \ar[d]^\cong & h^1(\mathcal{I}_{Z,\P^n}(m)) \ar[r] \ar[d] & 0 \\
h^0(\str_{S}(m)) \ar[r] & h^0(\str_Z(m)) \ar[r]  & h^1(\mathcal{I}_{Z,S}(m)) \ar[r] & 0 \\
}
\end{equation} 

\noindent A straightforward diagram chase shows that $h^1(\mathcal{I}_{Z,\P^n}(m)) = 0 $ if and only if $h^1(\mathcal{I}_{Z,S}(m))=0$. Thus, since $\omega_S = \str_S(k)$, Serre duality yields 
$$ \Ext^1( \mathcal{I}_{Z,S}(e+2), \str_S) \cong h^1(\mathcal{I}_{Z,S}(m))^{\vee} \neq 0$$
and similarly $$ \Ext^1( \mathcal{I}_{Z',S}(e+2), \str_S) = 0.$$ Therefore, by Theorem \ref{genserre}, we can find a nontrivial extension 
\begin{equation}
0 \to \str_S \to \E \to \mathcal{I}_{Z,S}(e+2) \to 0
\end{equation}
with $\E$ locally free.

For the sake of contradiction, assume $$\deg(Z) < (e+1) \cdot d_3 \cdot d_4 \cdot \cdots \cdot d_n = (e+1) \cdot \deg(S).$$
Then $$c_1(\E)^2 - 4c_2(\E) = (e+2)^2 \deg(S) - 4 \deg(Z) > ((e+2)^2  - 4 (e+1)) \deg(S) \geq 0.$$ Thus, by Theorem \ref{thm:Bog}, $\E$ is Bogomolov unstable, and we can write it as an extension
$$0 \to L \to \E \to M \otimes \mathcal{I}_W \to 0,$$
where $L$ and $M$ are line bundles satisfying the conditions from Theorem \ref{thm:Bog} and $W \subset S$ is a finite subscheme. Since the Picard group of $S$ is generated by the hyperplane class, we can set $L = \str_S(\lambda)$.

Following the proof of Lemma \ref{lem:key lemma} ($e+2$ taking the place of $\alpha$), we conclude that 
$$2 \lambda > e +2 \geq 2,$$ $$e+2 > \lambda, \text{\, and}$$ 
$$ \lambda(e+2-\lambda)  \cdot \deg(S) \leq \deg(Z) < (e+2) \cdot \deg (S).$$
Combining these inequalities, we obtain $$e+2-\lambda = 1.$$
This implies $$ (e+1)  \cdot \deg(S) \leq \deg(Z) ,$$
which gives us a contradiction.
\end{proof}

 \bibliographystyle{alpha}
\bibliography{bibliography.bib}

\end{document}